\documentclass[11pt]{amsart}

\usepackage{amsmath}
\usepackage{amsfonts}
\usepackage{amssymb}
\usepackage{amscd}
\usepackage{amsthm}
\usepackage{filecontents}
\usepackage{framed}
\usepackage{fullpage}
\usepackage{latexsym}

\usepackage{hyperref}
\hypersetup{colorlinks,linkcolor={red},citecolor={blue},urlcolor={blue}}

\evensidemargin\oddsidemargin

\newtheorem{theorem}{Theorem}[section]

\newtheorem{corollary}[theorem]{Corollary}
\newtheorem{proposition}[theorem]{Proposition}

\theoremstyle{definition}
\newtheorem{definition}[theorem]{Definition}
\newtheorem{assumption}[theorem]{Assumption}
\newtheorem{remark}[theorem]{Remark}

\numberwithin{equation}{section}

\newcommand{\CC}{\mathbb C}
\newcommand{\HH}{\mathbb H}
\newcommand{\NN}{\mathbb N}
\newcommand{\PP}{\mathbb P}
\newcommand{\QQ}{\mathbb Q}

\newcommand{\ZZ}{\mathbb Z}

\newcommand{\cD}{\mathcal D}
\newcommand{\cE}{\mathcal E}
\newcommand{\cH}{\mathcal H}
\newcommand{\SL}{\mathop{\mathrm {SL}}\nolimits}

\newcommand{\SO}{\mathop{\mathrm {SO}}\nolimits}

\newcommand{\Orth}{\mathop{\null\mathrm {O}}\nolimits}
\newcommand{\rank}{\mathop{\mathrm {rank}}\nolimits}
\newcommand{\latt}[1]{{\langle{#1}\rangle}}
\newcommand{\Kthree}{\mathop{\mathrm {K3}}\nolimits}

\newcommand{\II}{\operatorname{II}}

\newcommand{\im}{\operatorname{Im}}

\newcommand{\Proj}{\operatorname{Proj}}
\newcommand{\w}{\operatorname{w}}
\newcommand{\Div}{\operatorname{Div}}

\usepackage{longtable}

\usepackage{tikz}
\usetikzlibrary{positioning}
\usetikzlibrary{arrows}

\newenvironment{psmallmatrix}
  {\left(\begin{smallmatrix}}
  {\end{smallmatrix}\right)}

\begin{document}

\title[On some free algebras of orthogonal modular forms]{On some free algebras of orthogonal modular forms}

\author{Haowu Wang}

\address{Max-Planck-Institut f\"{u}r Mathematik, Vivatsgasse 7, 53111 Bonn, Germany}

\email{haowu.wangmath@gmail.com}

\author{Brandon Williams}

\address{Fachbereich Mathematik, Technische Universitat Darmstadt, 64289 Darmstadt, Germany}

\email{bwilliams@mathematik.tu-darmstadt.de}

\subjclass[2010]{11F50,11F55}

\date{\today}

\keywords{Symmetric domains of type IV, modular forms on orthogonal groups, Jacobi forms, root systems, free algebras}

\begin{abstract}
For 25 orthogonal groups of signature $(2,n)$ related to the root lattices $A_1$, $2A_1$, $3A_1$, $4A_1$, $A_2$, $A_3$, $A_4$, $A_5$, $A_6$, $A_7$, $D_4$, $D_5$, $D_6$, $D_7$, $D_8$, $E_6$, $E_7$,  we prove that the algebras of modular forms on symmetric domains of type IV are freely generated by the additive lifts of some special Jacobi forms. The proof is universal and elementary. 
\end{abstract}

\maketitle

\section{Introduction}

It is an interesting problem to determine the structure of the algebras of automorphic forms on symmetric domains $\cD$ of dimension greater than 2.  The simplest possible structure is that of a free algebra, but free algebras of modular forms are relatively rare. In general, the algebra of modular forms $M_*(\Gamma)$ for a congruence group $\Gamma$ acting on $\cD$ is freely generated if and only if the Satake-Baily-Borel compactification $(\cD / \Gamma)^*$ is a weighted projective space. Equivalently, if $\mathcal{A}$ denotes the affine cone over $\cD$, then $M_*(\Gamma)$ is free if and only if $(\mathcal{A}/\Gamma)^*$ is nonsingular. It is known, for example, that algebras of modular forms for orthogonal groups of signature $(2,n)$ never satisfy this condition when $n > 10$ \cite{SV17}.  There is a classification of free algebras of Hilbert modular forms in \cite{Stu19}, which shows that the free algebras of modular forms for orthogonal groups of signature (2,2) are very rare.

The first example of such an algebra was found by Igusa, who proved in \cite{Igu62} that the algebra of even-weight Siegel modular forms of genus 2 is freely generated by forms of weights 4, 6, 10, 12. Siegel modular forms of genus 2 can be realized as modular forms for the orthogonal group $\Orth(2,3)$. After Igusa, several other authors constructed free algebras of $\Orth(2,n)$-modular forms, e.g. \cite{AI05, Aok00, DK03, DK06, Klo05, Kri05, Woi18} using the theory of modular forms and \cite{FH00, FS07, HU14, Vin10, Vin13, Vin18} using more geometric methods. As a continuation of this type of work,  in this paper we prove that the spaces of modular forms for 25 orthogonal groups are free algebras in a universal way.  Our approach is new and gives a general rule to characterize the weights of generators.

\subsection{Orthogonal modular forms}
Orthogonal modular forms are automorphic forms on symmetric domains of type IV for orthogonal groups of signature $(2,n)$.  Let $M$ be an even lattice of signature $(2,n)$ with $n\geq 3$ and $M^\vee$ be its dual lattice.  The Hermitian symmetric domain of type IV is defined as 
\begin{equation*}
\cD(M)=\{[\mathcal{Z}] \in  \PP(M\otimes \CC):  (\mathcal{Z}, \mathcal{Z})=0, (\mathcal{Z},\bar{\mathcal{Z}}) > 0\}^{+},
\end{equation*}
where the symbol $+$ means that we choose one of the two connected components. Let $\Orth^+ (M) < \Orth(M)$ be the subgroup preserving the component $\cD(M)$.

\begin{definition}
Let $\Gamma$ be a finite index subgroup of $\Orth^+ (M)$.  A modular form of weight $k\in \NN$ for $\Gamma$ is a holomorphic function on the affine cone 
$$
\mathcal{A}(M) = \{\mathcal{Z} \in M \otimes \CC: \; [\mathcal{Z}] \in \cD(M))\}
$$ 
of $\cD(M)$ satisfying
\begin{align*}
F(t\mathcal{Z})&=t^{-k}F(\mathcal{Z}), \quad \forall t \in \CC^*,\\
F(g\mathcal{Z})&=F(\mathcal{Z}), \quad \forall g\in \Gamma.
\end{align*}
\end{definition}

The space of modular forms of weight $k$ is a finite-dimensional complex vector space and we denote it by $M_k(\Gamma)$. We know from \cite{BB66} that the graded algebra
$$
M_*(\Gamma)=\bigoplus_{k=0}^\infty M_k(\Gamma)
$$
is finitely generated over $\CC$.  Moreover, the projective variety $\Proj (M_*(\Gamma))$ coincides with $(\cD(M)/\Gamma)^*$ which is the Satake-Baily-Borel compactification of $\cD(M)/\Gamma$. In particular, if $M_*(\Gamma)$ is a free algebra generated by $n+1$ forms of weights $k_1$, $k_2$, ..., $k_{n+1}$,  then $(\cD(M)/\Gamma)^*$ is a weighted projective space with weights $(k_1, k_2, ..., k_{n+1})$.

\subsection{Main results} \label{subsec:main result}
We focus on the particular case $M=\II_{2,2}\oplus L(-1)$, where $\II_{2,2}=\II_{1,1}\oplus\II_{1,1}$, $\II_{1,1}$ is a hyperbolic plane, and $L$ is an even positive definite lattice. Let $\widetilde{\Orth}^+(M)$ denote the discriminant kernel, which is the kernel of the natural homomorphism $\Orth^+ (M) \to \Orth(M^\vee/M)$.  

We assume that $\widetilde{\Orth}^+(M)<\Gamma$. At the standard one-dimensional cusp determined by $\II_{2,2}$, the symmetric space $\cD(M)$ can be realized as the tube domain 
$$
\cH(L)=\{Z=(\tau,\mathfrak{z},\omega)\in \HH\times (L\otimes\CC)\times \HH: 
(\im Z,\im Z)>0\}, 
$$
where $(\im Z,\im Z)=2\im \tau \im \omega - (\im \mathfrak{z},\im \mathfrak{z})_L$. We consider the Fourier expansion of orthogonal modular forms on the above tube domain. The Fourier-Jacobi coefficients are in fact holomorphic Jacobi forms associated to the lattice $L$ (see e.x. \cite{CG13} for the theory of Jacobi forms in many variables).  We notice that there is a nice structure result due to Wirthm\"{u}ller about Jacobi forms. Wirthm\"{u}ller \cite{Wir92} proved that the spaces of weak Jacobi forms invariant under the Weyl group are free algebras for all irreducible root systems except $E_8$. This type of Jacobi forms is called Weyl invariant Jacobi forms. Let $R$ be an irreducible root system of rank $r$ and not of type $E_8$. Let $L_R$ and $W(R)$ denote the generated root lattice and  the Weyl group. When $L_R$ is an odd lattice, we equip $L_R$ with the normalized bilinear form $2(\cdot,\cdot) $ rescaled by 2. The spaces of Weyl invariant weak Jacobi forms of type $R$ is freely generated by $r+1$ forms $\phi_j$ over the ring of $\SL_2(\ZZ)$ modular forms. We denote the weight and index of $\phi_j$ by $-k_j$ and $m_j$ for $1\leq j\leq r+1$.  In view of the relation between Jacobi forms and orthogonal modular forms linked by the Fourier-Jacobi expansion, we guess that there should be some connection between free algebras of Jacobi forms and orthogonal modular forms. Our main result is the following theorem which illuminates the mysterious connection. 

\begin{theorem}\label{th:main}
Let $R$ be a root system of type $A_r (1\leq r \leq 7)$, $B_r (2\leq r \leq 4)$, $D_r (4\leq r \leq 8)$, $C_r (3\leq r \leq 8)$, $G_2$, $F_4$, $E_6$, or $E_7$. We define $\Gamma_R<\Orth^+(\II_{2,2}\oplus L_R(-1))$ as the subgroup generated by $\widetilde{\Orth}^+(\II_{2,2}\oplus L_R(-1))$ and $W(R)$.  Then the graded algebra $M_*(\Gamma_R)$ is freely generated by $r+3$ forms of weights $4$, $6$, and $-k_j+12m_j$,  $1\leq j\leq r+1$.
\end{theorem}

The groups $\Gamma_R$ and the weights of generators are listed in Theorem \ref{th:generators}. We explain the proof of the theorem. Firstly, the Fourier-Jacobi coefficients of any modular form for $\Gamma_R$ are Weyl invariant holomorphic Jacobi forms of type $R$.  Suppose there exist the following $r+3$ modular forms for $\Gamma_R$:
\begin{itemize}
\item[(a)] Two modular forms $\widetilde{E}_4$ and $\widetilde{E}_6$ of weights 4 and 6 whose first Fourier--Jacobi coefficients are respectively the Eisenstein series $E_4$ and $E_6$ on $\SL_2(\ZZ)$;

\item[(b)] $r+1$ modular forms $\Phi_j\in M_{-k_j+12m_j}(\Gamma_R)$ whose first nonzero Fourier-Jacobi coefficients are $\Delta^{m_j}\phi_j$ in their $m_j^{\mathrm{th}}$ term, where $\Delta$ is the normalized cusp form of weight 12 on $\SL_2(\ZZ)$.
\end{itemize}
Then we can kill the first Fourier-Jacobi coefficients of a given modular form by a polynomial combination of the above $r+3$ functions. If the first $n$ Fourier-Jacobi coefficients of a modular form are zero (for a certain $n$ which depends on the structure of the ring of weak Jacobi forms) then it is identically zero. We use this to conclude that the space $M_*(\Gamma_R)$ is generated by the above $r+3$ functions and hence a free algebra.  We are able to construct the modular forms of the above form for all root systems listed in the theorem. This completes the proof.

Some cases of Theorem \ref{th:main} are already known. The $A_1$ case is Igusa's result. The $A_2$ and $B_2$ cases were first proved in \cite{DK03, DK06} in the context of Hermitian modular forms. The $B_3$ and $C_3$ cases were proved in \cite{FH00} and in \cite{Klo05} by different methods. The $B_4$ case was studied in \cite{Kri11}. The $B_n$-tower was also considered systematically in \cite{Woi18}. Krieg \cite{Kri05} worked out the $D_4$ case in the context of quaternionic modular forms. Making use of the interpretation of orthogonal modular varieties as moduli spaces of lattice polarized K3 surfaces, Vinberg determined the structure of orthogonal modular forms for the $D_n$-tower in \cite{Vin10, Vin18}, but Vinberg does not construct generators. In this paper we prove that the additive lifts of Jacobi Eisenstein series are generators in the most complicated cases $D_8$, $E_6$, and $E_7$. To the authors' knowledge the cases $A_4, A_5, A_6, A_7, E_6, E_7$ are new. 

As a direct consequence, we obtain the modularity of formal Fourier--Jacobi expansions for all cases in our theorem (see Corollary \ref{cor:modularity}). This property is useful to prove that certain arithmetic generating series are modular forms and is known to hold for Siegel modular forms \cite{BR15}. For orthogonal modular forms it was previously only known in the $A_1$ case (see \cite{Aok00}).

The layout of this paper is as follows. In \S \ref{sec:Jacobi forms} we introduce Weyl invariant Jacobi forms and recall Wirthm\"{u}ller's theorem. In \S \ref{sec:upper bound} we estimate the upper bound of the dimension of orthogonal modular forms using Jacobi forms. In \S \ref{sec:generators} we explain why the existence of modular forms of types (a) and (b) implies the above theorem. The existence of the desired modular forms is proved in \S \ref{sec:free algebras}, which completes the proof of our main theorem.

\section{Weyl invariant Jacobi forms}\label{sec:Jacobi forms}
In this section we introduce Weyl invariant Jacobi forms and Wirthm\"{u}ller's structure result. Let $R$ be an irreducible root system of rank $r$. The classification of $R$ is as follows  (see \cite{Bou60})
\begin{align*}
&A_n (n\geq 1),& &B_n (n\geq 2),& &C_n (n\geq 3),& &D_n (n\geq 4),& &E_6,& &E_7,& &E_8,& &G_2,& &F_4.&
\end{align*}
Let $L_R$ and $W(R)$ be the root lattice and Weyl group generated by $R$ respectively. When $L_R$ is an odd lattice, we rescale its bilinear form by $2$. Thus $L_R$ is always an even positive definite lattice and we denote its (rescaled) bilinear form by $\latt{\cdot, \cdot}$.  

Weyl invariant Jacobi forms are Jacobi forms of lattice index that are invariant under the action of the Weyl group of $R$ on the abelian variable. These Jacobi forms appear in the Fourier--Jacobi expansions of orthogonal modular forms. We refer to \cite{Gri94, CG13} for the theory.

\begin{definition}
Let $\varphi : \HH \times (L_R \otimes \CC) \rightarrow \CC$ be a holomorphic function and $k\in \ZZ$, $t\in \NN$. If $\varphi$ satisfies the following properties
\begin{align*}
&\varphi(\tau, \sigma(\mathfrak{z}))=\varphi(\tau, \mathfrak{z}), \quad \sigma\in W(R),\\
&\varphi (\tau, \mathfrak{z}+ x \tau + y)= \exp\left(-t\pi i [ \latt{x,x}\tau +2\latt{x,\mathfrak{z}} ]\right) \varphi ( \tau, \mathfrak{z} ), \quad x,y\in L_R,\\
&\varphi \left( \frac{a\tau +b}{c\tau + d},\frac{\mathfrak{z}}{c\tau + d} \right) = (c\tau + d)^k \exp\left( t\pi i \frac{c\latt{\mathfrak{z},\mathfrak{z}}}{c \tau + d}\right) \varphi ( \tau, \mathfrak{z} ), \quad \left( \begin{array}{cc}
a & b \\ 
c & d
\end{array} \right)   \in \SL_2(\ZZ),
\end{align*}
and if its Fourier expansion takes the form 
\begin{equation*}
\varphi ( \tau, \mathfrak{z} )= \sum_{ n=0}^{\infty}\sum_{ \ell \in L_R^\vee}f(n,\ell)e^{2\pi i (n\tau + \latt{\ell,\mathfrak{z}})},
\end{equation*}
then $\varphi$ is called a $W(R)$-invariant weak Jacobi form of weight $k$ and index $t$. If $f(n,\ell) = 0$ whenever $2nt - \latt{\ell,\ell} <0$, then $\varphi$ is called a $W(R)$-invariant holomorphic Jacobi form.  
We denote by $J^{\w ,W(R)}_{k,L_R,t}$ and  $ J^{W(R)}_{k,L_R,t} $ the vector spaces of $W(R)$-invariant weak and holomorphic  Jacobi forms of weight $k$ and index $t$, respectively.
\end{definition}

We introduce some notations appearing in Wirthm\"{u}ller's theorem.  The dual root system of $R$ is defined as 
\begin{equation*}
R^\vee=\{ r^\vee: r\in R \},
\end{equation*}
where $r^\vee=\frac{2}{(r,r)}r$ is the coroot of $r$.
Let $\widetilde{\alpha}$ denote the highest root of $R^\vee$. All $W(R)$-invariant weak Jacobi forms form a bigraded ring graded by weight and index
$$
J^{\w ,W(R)}_{*,L_R,*}=\bigoplus_{t\in \NN}\bigoplus_{k\in\ZZ} J^{\w ,W(R)}_{k,L_R,t}.
$$
Let $M_*(\SL_2(\ZZ))$ be the graded ring of modular forms for $\SL_2(\ZZ)$. In 1992,  Wirthm\"{u}ller  proved the following theorem.

\begin{theorem}[see Theorem 3.6 in \cite{Wir92}]
If $R\neq E_8$, then $J^{\w ,W(R)}_{*,L_R,*}$ over $M_*(\SL_2(\ZZ))$ is freely generated by $r+1$ $W(R)$-invariant weak Jacobi forms of weight $-k_j$ and index $m_j$
\begin{equation*}
\phi_{-k_j,R,m_j}(\tau,\mathfrak{z}), \quad 1\leq j\leq r+1.
\end{equation*}
Apart from $(k_1, m_1)=(0, 1)$, the indices $m_j$ are the coefficients of $\widetilde{\alpha}^\vee$  written as a linear combination of the simple roots of $R$. The integers $k_j$ are the degrees of the generators of the ring of $W(R)$-invariant polynomials, namely the exponents of the Weyl group $W(R)$ increased by $1$.
\end{theorem}
 
We formulate the weights and indices of these generators in Table \ref{tablewi}.  The generators for root systems of types $A_n$, $B_n$ and $D_4$ were constructed  in \cite{Ber99}. The construction of generators for root systems $E_6$ and $E_7$ was given in \cite{Sak19}. We refer to \cite{AG19} for the generators of type $D_n$ with $1\leq n \leq 8$.  We remark that it was proved in \cite{Wan18} that $J^{\w ,W(E_8)}_{*,E_8,*}$ is not a polynomial algebra.

\begin{table}[ht]
\caption{Weights and indices of generators of Weyl invariant weak Jacobi forms ($B_n: n\geq 2$, $C_n: n\geq 3$, $D_n: n\geq 4$)}\label{tablewi}
\renewcommand\arraystretch{1.5}
\noindent\[
\begin{array}{|c|c|c|c|}
\hline 
R & L(R) & W(R) & (k_j,m_j) \\ 
\hline 
A_n & A_n & W(A_n) & (0,1), (s,1) : 2\leq s\leq n+1\\ 
\hline 
B_n & nA_1 & \Orth(nA_1) & (2s,1) : 0\leq s \leq n  \\ 
\hline 
C_n & D_n & W(C_n) & (0,1), (2,1), (4,1), (2s,2) : 3\leq s \leq n  \\ 
\hline 
D_n & D_n & W(D_n) &  (0,1), (2,1), (4,1), (n,1), (2s,2) : 3\leq s \leq n-1 \\ 
\hline
E_6 & E_6 & W(E_6) & (0,1), (2,1), (5,1), (6,2), (8,2), (9,2), (12,
3)  \\ 
\hline
E_7 & E_7 & W(E_7) & (0,1), (2,1), (6,2), (8,2), (10,2), (12,
3), (14,3), (18,4)  \\ 
\hline
G_2 & A_2 & \Orth(A_2) & (0,1), (2,1), (6,2)  \\ 
\hline
F_4 & D_4 & \Orth(D_4) &  (0,1), (2,1), (6,2), (8,2), (12,3) \\ 
\hline
\end{array} 
\]
\end{table}

\section{An upper bound of the dimension of orthogonal modular forms}\label{sec:upper bound}
In \cite{Aok00}, Aoki estimated the dimension of Siegel modular forms of genus 2 using the theory of classical Jacobi forms due to Eichler and Zagier \cite{EZ85}. We extend his idea to the general case of orthogonal modular forms. We give an upper bound of the dimension of orthogonal modular forms in terms of the dimensions of Jacobi forms, which will be used later.

Let $M=\II_{2,2}\oplus L(-1)$ and $\Gamma=\latt{\widetilde{\Orth}^+(M), W}$, where $W$ is a subgroup of $\Orth(L)$ containing $\widetilde{\Orth}(L)$. Let $F$ be a modular form of weight $k$ with respect to $\Gamma$ with the trivial character. We consider its Fourier and Fourier-Jacobi expansions
\begin{align*}
F(\tau,\mathfrak{z},\omega)&=\sum_{\substack{n,m\in \NN, \ell\in L^\vee \\2nm-(\ell,\ell)\geq 0}}f(n,\ell,m)q^n\zeta^\ell\xi^m\\
&=\sum_{m=0}^{\infty}\phi_m(\tau,\mathfrak{z})\xi^m,
\end{align*}
where $q=\exp(2\pi i\tau)$, $\zeta^\ell=\exp(2\pi i (\ell, \mathfrak{z}))$, $\xi=\exp(2\pi i \omega)$.
Then $\phi_m\in J_{k,L,m}^{W}$, i.e. $\phi_m$ is a $W$-invariant holomorphic Jacobi form of weight $k$ and index $m$ associated to the lattice $L$. Moreover, we have the symmetric relation
$$
f(n,\ell,m)=f(m,\ell,n),\quad \forall (n,\ell,m)\in \NN \oplus L^\vee \oplus \NN. 
$$
 For $r\geq 1$, we define
$$
M_k(\Gamma)(\xi^r)=\{F\in M_k(\Gamma) : \phi_m=0,\; \text{for all $m<r$}\}
$$
and 
$$ 
J_{k,L,m}^{W}(q^r)=\{\phi \in J_{k,L,m}^{W}: \phi=O(q^r) \},
$$
here $\phi=O(q^r)$ means that its Fourier coefficients $f(n,\ell)=0$ for all $n< r$ and $\ell\in L^\vee$.  For convenience, we also set $M_k(\Gamma)(\xi^0)=M_k(\Gamma)$ and $J_{k,L,m}^{W}(q^0)=J_{k,L,m}^{W}$.

It is easy to check that the following sequence is exact:
$$ 
0\longrightarrow M_k(\Gamma)(\xi^{r+1})\longrightarrow  M_k(\Gamma)(\xi^r)\stackrel{P_r}\longrightarrow J_{k,L,r}^{W}(q^r),
$$
where the map $P_r$ sends $F$ to its Fourier--Jacobi coefficient $\phi_r$.
From this, we deduce the inequality
$$ 
\dim M_k(\Gamma)(\xi^{r})-\dim M_k(\Gamma)(\xi^{r+1}) \leq \dim J_{k,L,r}^{W}(q^r). 
$$
Since $\dim M_k(\Gamma) < \infty$, it follows that $ M_k(\Gamma)(\xi^{r})=\{ 0\} $ for sufficiently large $r$. If $\phi\in J_{k,L,m}^{W}(q^r)$, then $\phi / \Delta^r$ is a weak Jacobi form of weight $k-12r$ and index $m$. It follows that 
$$
\dim J_{k,L,m}^{W}(q^r) \leq \dim J_{k-12r,L,m}^{\w, W}.
$$
We then deduce
\begin{equation}\label{eq:MF-JF}
\dim M_k(\Gamma) \leq \sum_{r=0}^{\infty} \dim J_{k,L,r}^{W}(q^r)
\leq \sum_{r=0}^{\infty} \dim J_{k-12r,L,r}^{\w ,W}.
\end{equation}
In some particular cases, the last sum in \eqref{eq:MF-JF} is a finite sum. More precisely, if there exists a positive constant $\delta$ less than 12 such that the Jacobi forms associated to $L$ satisfy the condition
\begin{equation}
J_{k,L,m}^{\w,W}=\{0\} \quad \text{if} \quad k<-\delta m,
\end{equation}
then \eqref{eq:MF-JF} can be improved as follows: 
\begin{equation}
\dim M_k(\Gamma) \leq \sum_{r=0}^{\frac{k}{12-\delta}} \dim J_{k-12r,L,r}^{\w ,W}.
\end{equation}
It is enough to check this condition on a system of generators for $J^{\w ,W}_{*,L,*}$. For example,  by Table \ref{tablewi}, we can take $\delta = 5$ in the case of $W(E_7)$-invariant Jacobi forms.

\section{The shape of generators of free algebras}\label{sec:generators}
In this section, we introduce a universal and elementary method to prove that the space of orthogonal modular forms is a free algebra when the associated bigraded ring of Jacobi forms is free and there exist certain special modular forms (i.e. the modular forms of types (a) and (b) in section \ref{subsec:main result}).

Let $W$ be a subgroup of $\Orth(L)$ containing $\widetilde{\Orth}(L)$. We assume that the bigraded ring $J_{*,L,*}^{\w,W}$ is freely generated as a $M_*(\SL_2(\ZZ))$-algebra by the forms
$$
\phi_i:=\phi_{-k_i,m_i}\in J_{-k_i,L,m_i}^{\w,W}, \;1\leq i \leq \rank(L)+1.
$$
Let $M=\II_{2,2}\oplus L(-1)$ and define the congruence subgroup $\Gamma=\latt{\widetilde{\Orth}^+(M), W}$. We formulate the following assumption.

\begin{assumption}\label{assum}
There exist modular forms $\widetilde{E}_4$, $\widetilde{E}_6$ and $\Phi_i$, $1 \leq i \leq \rank(L)+1$ satisfying the following conditions:
\begin{enumerate}
\item $\widetilde{E}_4 \in M_4(\Gamma)$ and $\widetilde{E}_6\in M_6(\Gamma)$, and their first Fourier--Jacobi coefficients are respectively the Eisenstein series $E_4$ and $E_6$ on $\SL_2(\ZZ)$.
\item $\Phi_i\in M_{-k_i+12m_i}(\Gamma)$ has Fourier--Jacobi expansion $$\Phi_i = (\Delta^m \phi_i) \xi^{m_i} + O(\xi^{m_i+1}).$$
\end{enumerate}
\end{assumption}

\begin{proposition}
If Assumption \ref{assum} holds, then $M_*(\Gamma)$ is a free algebra over $\CC$ generated by $\widetilde{E}_4$, $\widetilde{E}_6$ and $\Phi_i$, $1\leq i \leq \rank(L)+1$.
\end{proposition}

\begin{proof}
Let $F=\sum_{m=0}^{\infty} f_m \xi^m \in M_k(\Gamma)$. Then $f_0\in M_{k}(\SL_2(\ZZ))$ and there exists a polynomial $P_0$ over $\CC$ in two variables such that $f_0=P_0(E_4,E_6)$. It follows that
$$
F_1:=F-P_0(\widetilde{E}_4,\widetilde{E}_6)=\sum_{m=1}^{\infty} f_{m,1} \xi^m \in M_k(\Gamma)(\xi^1),
$$
and $f_{m,1}\in J_{k,L,m}^{W}(q)$ for $m\geq 1$. 

Now let $r\geq 1$ and suppose a form $F_r=\sum_{m=r}^{\infty} f_{m,r} \xi^m \in M_k(\Gamma)(\xi^r)$ is given. Since $f_{r,r}\in J_{k,L,r}^{W}(q^r)$, it follows that $f_{r,r}/ \Delta^r\in J_{k-12r,L,r}^{\w,W}$. Therefore there exists a polynomial $P_r$ over $\CC$ in $\rank(L)+3$ variables such that $f_{r,r}=\Delta^r P_r(E_4,E_6,\phi_i)$, which yields 
$$
F_{r+1}:=F_r-P_r(\widetilde{E}_4,\widetilde{E}_6,\Phi_i)=\sum_{m=r+1}^{\infty} f_{m,r+1} \xi^m \in M_k(\Gamma)(\xi^{r+1}),
$$
and $f_{m,r+1}\in J_{k,L,m}^W(q^{r+1})$ for $m\geq r+1$.
Indeed, let 
$$
\varphi=E_4^aE_6^b\prod_{i=1}^{\rank(L)+1} \phi_i^{c_i} 
$$
be any monomial in $P_r$. By considering the weight and index of $\varphi$, we find
$$
4a+6b-\sum_{i=1}^{\rank(L)+1} k_ic_i = k-12r, \quad \sum_{i=1}^{\rank(L)+1} m_ic_i = r.
$$
We conclude that the first nonzero Fourier-Jacobi coefficient of
$$
\widetilde{E}_4^a\widetilde{E}_6^b\prod_{i=1}^{\rank(L)+1} \Phi_i^{c_i} 
$$ 
is exactly  $\Delta^r \varphi$. Therefore the $r^{\text{th}}$  Fourier--Jacobi coefficient of $F_{r+1}$ is zero.

The proof follows by induction on $r$ because $M_k(\Gamma)(\xi^r)$ is trivial when $r$ is sufficiently large.  The algebraic independence of the $r+3$ generators over $\CC$ follows from the algebraic independence of the $r+1$ basic weak Jacobi forms over $M_*(\SL_2(\ZZ))$.
\end{proof}

From the above proof, it is easy to derive the following corollary.

\begin{corollary}\label{cor:structure}
Suppose that Assumption \ref{assum} holds. For any weak Jacobi form $\phi\in J_{k,L,m}^{\w,W}$, there exists a modular form of weight $k+12m$ for $\Gamma$ whose first nonzero Fourier-Jacobi coefficient is $(\Delta^m\phi)\cdot \xi^m$. Moreover, we have the equality
$$
\dim M_k(\Gamma) = \sum_{r=0}^{\infty} \dim J_{k-12r,L,r}^{\w ,W}.
$$
In particular, $J_{k-12r,L,r}^{\w ,W}=\{0\}$ for sufficiently large $r$.
\end{corollary}

Let $k$ be a positive integer. A formal series of holomorphic  Jacobi forms is an element
$$
\Psi(Z)=\sum_{m=0}^{\infty} \psi_m \xi^m \in \prod_{m=0}^\infty J_{k,L,m}^{W}.
$$
We call $\Psi$ a formal Fourier-Jacobi expansion of weight $k$ if it satisfies 
$$
f_m(n,\ell)=f_n(m,\ell), \quad m,n\in \NN, \ell \in L^\vee,
$$
where $f_m(n,\ell)$ are Fourier coefficients of $\psi_m$. We denote the space of such expansions by $FM_k(\Gamma)$.  As a direct consequence, we obtain the modularity of formal Fourier-Jacobi expansions:

\begin{corollary}\label{cor:modularity}
Suppose that Assumption \ref{assum} holds. Then $FM_k(\Gamma)=M_k(\Gamma)$ for any $k\in \NN$. In other word, every formal Fourier-Jacobi expansion is convergent on the tube domain $\cH(L)$ and defines an orthogonal modular form.
\end{corollary}

\begin{proof}
Firstly, the Fourier-Jacobi expansion of modular forms gives the following injective map
$$
M_k(\Gamma) \to FM_k(\Gamma), \quad F\mapsto \text{Fourier-Jacobi expansion of $F$}. 
$$
Using a similar argument as in the previous subsection, we get $\dim FM_k(\Gamma) \leq \sum_{r=0}^{\infty} \dim J_{k-12r,L,r}^{\w ,W}$.   We then prove the surjectivity of the above map by Corollary \ref{cor:structure}.
\end{proof}

We remark that the modularity of formal Fourier-Jacobi expansions is known for Siegel modular forms \cite{BR15}. The result of \cite{BR15} is a key step in the proof that certain generating series from arithmetic geometry are in fact modular forms.

\section{Free algebras related to root systems}\label{sec:free algebras}
In this section we prove our main theorem. In fact, we show that all cases of Theorem \ref{th:main} satisfy Assumption \ref{assum}.  We use the additive (Gritsenko) lift, i.e. the generalization of the Saito--Kurokawa lift to higher-rank lattices, to construct orthogonal modular forms. This sends a holomorphic Jacobi form $\phi$ of weight $k$ and index 1 associated to a lattice $L$ to a modular form of weight $k$ on $\widetilde{\Orth}^+(\II_{2,2}\oplus L(-1))$. If the constant term $f(0,0)$ of $\phi$ is not zero, then the first Fourier--Jacobi coefficient of the additive lift of $\phi$ is the $\SL_2(\ZZ)$-Eisenstein series of weight $k$. If $f(0,0)=0$, then the first nonzero Fourier--Jacobi coefficient of the additive lift is $\phi$ itself. We refer to \cite[Theorem 3.1]{Gri94} (or \cite[Theorem 3.2]{CG13}) for more details of additive lifts.

\subsection{The cases \texorpdfstring{$A_n$}{An}, \texorpdfstring{$B_n$}{Bn}, \texorpdfstring{$G_2$}{G2} and \texorpdfstring{$C_3$}{C3}}
We first consider root systems $A_n$ and $B_n$. In these cases, all generators of Weyl invariant Jacobi forms have index 1. Therefore it is easy to construct the desired orthogonal modular forms in Assumption \ref{assum}.

\begin{enumerate}
\item Let $R=B_n$ with $2\leq n\leq 4$.  The generated root lattice is the odd lattice $\ZZ^n$. After rescaling the bilinear form by 2,  $L_R$ becomes the root lattice $nA_1$.  

It is clear that $W(B_n)=\Orth(nA_1)$. The natural homomorphism $\Orth(nA_1)\to \Orth(nA_1^\vee / nA_1)$ is surjective. Recall that $\widetilde{\Orth}^+(\II_{2,2}\oplus nA_1)$ is the kernel of the surjective homomorphism $\Orth^+(\II_{2,2}\oplus nA_1)\to \Orth(nA_1^\vee / nA_1)$. Thus $\Orth^+(\II_{2,2}\oplus nA_1)$ is generated by $\widetilde{\Orth}^+(\II_{2,2}\oplus nA_1)$ and $\Orth(nA_1)$, which yields $\Gamma_R=\Orth^+(\II_{2,2}\oplus nA_1(-1))$.  This argument holds for all root lattices in Theorem \ref{th:main}. This helps us to calculate the groups $\Gamma_R$.

The Fourier--Jacobi coefficients of modular forms on $\Gamma_R$ are $W(B_n)$-invariant Jacobi forms. We know from Table \ref{tablewi} that $J_{*,nA_1,*}^{\w,\Orth(nA_1)}$ is a polynomial algebra over $M_*(\SL_2(\ZZ))$ and all generators $\phi_i$ have index $1$. The additive lifts of the Jacobi Eisenstein series of weights $4$ and $6$ give the modular forms $\widetilde{E_4}$ and $\widetilde{E_6}$. It is easy to check that $\Delta\phi_i$ are always holomorphic  Jacobi forms. The additive lifts of $\Delta\phi_i$ give the modular forms $\Phi_i$. Thus Assumption \ref{assum} is satisfied. 

\item Let $R=A_n$ with $1\leq n\leq 7$.  In these cases, $\Gamma_R=\widetilde{\Orth}^+(\II_{2,2}\oplus A_n(-1))$ because $W(A_n)$ is contained in the discriminant kernel $\widetilde{\Orth}(A_n)$. We note here that $\widetilde{\Orth}^+(\II_{2,2}\oplus A_1(-1))=\Orth^+(\II_{2,2}\oplus A_1(-1))$. It is known that $J_{*,A_n,*}^{\w,W(A_n)}$ is a polynomial algebra over $M_*(\SL_2(\ZZ))$ and all generators have index $1$. Similarly, we conclude that Assumption \ref{assum} holds. 

\item Let $R=G_2$. Then $L_R=A_2$, $W(R)=\Orth(A_2)$, and $\Gamma_R=\Orth^+(\II_{2,2}\oplus A_2(-1))$. The Fourier--Jacobi coefficients of modular forms on $\Orth^+(\II_{2,2}\oplus A_2(-1))$ are $W(G_2)$-invariant Jacobi forms. We know that $J_{*,A_2,*}^{\w,\Orth(A_2)}$ is a polynomial algebra over $M_*(\SL_2(\ZZ))$ and the generators are $\phi_{0,A_2,1}$, $\phi_{-2,A_2,1}$ and $\phi_{-6,A_2,2}$. We notice that $\phi_{-6,A_2,2}$ can be chosen as $\phi_{-3,A_2,1}^2$, where $\phi_{-3,A_2,1}$ is a generator of $J_{*,A_2,*}^{\w,W(A_2)}$. Thus the desired modular form $\Phi_3$ corresponding to $\phi_{-6,A_2,2}$ can be constructed as the square of the additive lift of $\Delta\phi_{-3,A_2,1}$. Therefore Assumption \ref{assum} is satisfied.

\item Let $R=C_3$. Then $L_R=A_3\cong D_3$, $W(R)=\Orth(A_3)$, and $\Gamma_R=\Orth^+(\II_{2,2}\oplus A_3(-1))$. We know that $J_{*,A_3,*}^{\w,\Orth(A_3)}$ is a polynomial algebra over $M_*(\SL_2(\ZZ))$ generated by $\phi_{0,A_3,1}$, $\phi_{-2,A_3,1}$, $\phi_{-6,A_3,2}$ and $\phi_{-4,A_3,1}$. Similarly, $\phi_{-6,A_3,2}$ can be chosen as $\phi_{-3,A_3,1}^2$, where $\phi_{-3,A_3,1}$ is a generator of $J_{*,A_3,*}^{\w,W(A_3)}$. In a similar way, we prove that Assumption \ref{assum} holds.
\end{enumerate}

\subsection{The cases \texorpdfstring{$C_n$}{Cn}, \texorpdfstring{$D_n$}{Dn},  and \texorpdfstring{$F_4$}{F4}}\label{subsec:D}
Let $4\leq n\leq 8$. For $R=C_n$ or $D_n$, we have $L_R=D_n$. When $R=F_4$, we have $L_R=D_4$. It is known that $W(D_n)$ is a subgroup of index 2 of $W(C_n)$ and that $W(C_n)/W(D_n)$ is generated by the sign change for an odd number of coordinates of $D_n\otimes \CC$. It is easy to check that  
\begin{align*}
&\Gamma_{D_n}=\widetilde{\Orth}^+(\II_{2,2}\oplus D_n(-1)), \quad 4\leq n \leq 8,\\
&\Gamma_{C_n}=\Orth^+(\II_{2,2}\oplus D_n(-1)), \quad 5\leq n \leq 8,\\
&\Gamma_{C_4}=\Orth_1^+(\II_{2,2}\oplus D_4(-1)):=\latt{\widetilde{\Orth}^+(\II_{2,2}\oplus D_4(-1)), \text{odd sign change of coordinates}},\\
&\Gamma_{F_4}=\Orth^+(\II_{2,2}\oplus D_4(-1)).
\end{align*}

In these cases, the rings of Weyl-invariant Jacobi forms require generators of index larger than 1. There is no direct way to construct the desired modular forms in Assumption \ref{assum} corresponding  to basic Jacobi forms of index larger than 1. To overcome this,  we first construct the desired generators for the root system $C_8$ and then use the pullback trick to construct the desired generators for other $C_n$, $D_n$ and $F_4$. In order to construct the desired generators for $C_8$, we first prove that the exact dimension of the space of orthogonal modular forms coincides with the upper bound given in \S \ref{sec:upper bound} when the weight is low.  To prove this, we construct some linearly independent modular forms as polynomials in the additive lifts of Jacobi Eisenstein series, which gives a lower bound for the dimension.  We then conclude the existence of the desired generators using an idea similar to Corollary \ref{cor:structure}. 

\subsubsection{The case \texorpdfstring{$C_8$}{C8}}
We recall that $W(C_8)=\Orth(D_8)$ and $\Gamma_{C_8}=\Orth^+(\II_{2,2}\oplus D_8(-1))$. There is a well-known isomorphism between the spaces of Jacobi forms and vector-valued modular forms for the Weil representation of $\SL_2(\ZZ)$, which is given by the theta decomposition (see \cite[Lemma 2.3]{Gri94}). This isomorphism will be used very often later.  We see from the theta decomposition that every Jacobi form of index 1 for $D_8$ is automatically invariant under $W(D_8)$, namely $J_{k,D_8,1}=J_{k,D_8,1}^{W(D_8)}$. The additive lift of any $\Orth(D_8)$-invariant holomorphic Jacobi forms of index 1 for $D_8$ will be a modular form on $\Gamma_{C_8}$. 

The Fourier expansions of $\Orth(D_8)$-invariant Jacobi Eisenstein series of lattice index $D_8$ (i.e. index 1 for $D_8$) are straightforward to compute as follows. The root lattice $D_8$ is stably equivalent to the rescaled hyperbolic plane $\mathrm{II}_{1,1}(2)$ in the sense that there are unimodular lattices $U_1,U_2$ for which $U_1 \oplus D_8 \cong U_2 \oplus \mathrm{II}_{1,1}(2)$, so there is an isomorphism of graded $M_*(\mathrm{SL}_2(\mathbb{Z}))$-modules $$J_{*,D_8,1} \cong M_{*-4}(\rho)$$ where $M_*(\rho)$ denotes modular forms which transform with respect to the (dual) Weil representation attached to $\mathrm{II}_{1,1}(2)$. We write such modular forms in terms of their components: $$F = (F_{00}, F_{01}, F_{10}, F_{11}).$$ By a general result for Weil representations attached to rescaled hyperbolic planes (Proposition 3.4 of \cite{Car12}) there is an isomorphism $$M_*(\rho) \cong M_*(\mathrm{SL}_2(\mathbb{Z})) \oplus M_*(\Gamma_0(2)),$$ under which a pair of modular forms $f_1 \in M_k(\mathrm{SL}_2(\mathbb{Z}))$, $f_2 \in M_k(\Gamma_0(2))$ corresponds to the form $$F_{00} = (f_1 + f_2)/2, \; F_{01} = (f_1 - f_2)/2, \; F_{10} = (f_2|_k S + f_2 |_k U)/2, \; F_{11} = (F_2|_k S - F_2|_k U)/2,$$ where $S = \begin{psmallmatrix} 0 & -1 \\ 1 & 0 \end{psmallmatrix}$ and $U = \begin{psmallmatrix} 0 & -1 \\ 1 & 1 \end{psmallmatrix}$.

A Jacobi form in $J_{k,D_8,1}$ is $\Orth(D_8)$-invariant if and only if the corresponding modular form $F$ has equal components $F_{01} = F_{10}$, and this is true if and only if the forms $(f_1,f_2)$ as above satisfy $$f_1 = f_2 + f_2|_k S + f_2|_k U.$$ In this way we obtain a natural isomorphism of graded $M_*(\SL_2(\ZZ))$-modules $$J_{*,D_8,1}^{\Orth(D_8)} \cong M_{*-4}(\Gamma_0(2))$$ which respects Fourier coefficients (on the right-hand side involving both cusps) and identifies the two $\Orth(D_8)$-invariant Jacobi Eisenstein series of weights no less than 8 with the usual Eisenstein series for $\Gamma_0(2)$.

Unfortunately, the lifts to orthogonal modular forms have unwieldy Fourier expansions (in ten variables) and taking algebraic expressions in them to any significant precision involves an unmanageable number of operations. It is inconvenient to work with these modular forms directly. Our method is to instead restrict these lifts to certain embedded Siegel upper half-spaces which correspond to Grassmannians of sublattices of $\II_{2,2} \oplus D_8(-1)$ of signature $(2,3)$. These restrictions (called \emph{pullbacks}) can be computed without first computing the lifts due to the commutative diagrams

\begin{center} \begin{tikzpicture}[node distance=2.5cm, auto]
\node (A) {$J_{k,D_8,1}^{\Orth(D_8)}$};
\node (AA) [right of=A] {$ $};
\node (B) [right of=AA] {$J_{k,m}$};
\node (C) [below of=A] {$M_k(\Gamma_{C_8})$};
\node (D) [below of=B] {$M_k(K(m))$};
\draw[->] (B) to node {$\mathrm{lift}$} (D);
\draw[->] (A) to node {$\mathrm{lift}$} (C);
\draw[->] (A) to node {$e_v$} (B);
\draw[->] (C) to node {$\mathrm{pullback}$} (D);
\end{tikzpicture}
\end{center}
where $v \in D_8$ is any nonzero lattice vector of norm $m$, and $e_v$ denotes the map $$e_v : J_{k,D_8,1} \longrightarrow J_{k,m}, \; \; (e_v f)(\tau,z) = f(\tau,z \cdot v),$$ and where $K(m)$ is the paramodular group of degree two and level $m$, which is isomorphic to the group $\widetilde{\SO}^+(\II_{2,2}\oplus \latt{-2m})$. Any algebraic relation among the orthogonal Eisenstein series for $D_8$ will be mapped to an algebraic relation among the pullbacks. By making a fortunate choice of $v$, we hope to find that the pullbacks satisfy no relations in small weight and conclude that the orthogonal Eisenstein series span the whole spaces of modular forms of small weights.

The orthogonal modular variety attached to $\Orth^+(\II_{2,2} \oplus D_8(-1))$ has two zero-dimensional cusps, corresponding to the cosets of $D_8$ in its dual lattice which have integer norm modulo the orthogonal group. In particular, for every even $k \ge 8$, we obtain two orthogonal Eisenstein series $\cE_{k,0}$ and $\cE_{k,1}$ (the additive lifts of the Jacobi Eisenstein series) corresponding to the cusps of the zero and nonzero integer-norm cosets of $D_8$, respectively. In addition, we let $\cE_4$ and $\cE_6$ be the additive lifts of the holomorphic Jacobi forms which correspond to $1 \in M_0(\Gamma_0(2))$ and $2 E_2(\tau) - E_2(2\tau) \in M_2(\Gamma_0(2))$ as above.

\begin{proposition} 
The series $$\cE_4,\cE_6,\cE_{8,0},\cE_{8,1},\cE_{10,0},\cE_{10,1},\cE_{12,0},\cE_{12,1},\cE_{14,0},\cE_{16,0},\cE_{18,0}$$ satisfy no algebraic relations in weights less than $20$.
\end{proposition}
\begin{proof} We realized the $D_8$ lattice concretely as $\ZZ^8$ with bilinear form given by the Cartan matrix $$\begin{psmallmatrix} 2 & -1 & 0 & 0 & 0 & 0 & 0 & 0 \\ -1 & 2 &-1 & 0 & 0 & 0 & 0 & 0 \\ 0 & -1 &2 & -1 & 0 & 0 & 0 & 0 \\ 0 & 0  & -1 & 2 & -1 & 0 & 0 & 0 \\ 0 & 0 & 0 & -1 & 2 & -1 & 0 & 0 \\ 0 & 0 &0  & 0 & -1 & 2 & - 1 & -1 \\ 0 & 0 & 0 & 0 & 0 & -1 & 2 & 0 \\ 0 & 0 & 0 & 0 & 0 & -1 & 0 & 2 \end{psmallmatrix},$$ and we restricted the Jacobi Eisenstein series to the complex span of the vector $$v = (4,2,3,4,1,3,2,4)$$ of norm $24$. (The choice of $v$ is more or less arbitrary; roughly speaking, with larger-norm vectors the computation is more likely to succeed and with smaller-norm vectors the computation is faster, and we look for a balance between the two.) We computed enough Fourier coefficients of the pulled-back Jacobi Eisenstein series $f(\tau,z) = E_{*,*}(\tau,z)$ to identify them uniquely, and used faster methods for Jacobi forms of scalar index to compute their Gritsenko lifts to the paramodular group $K(24)$. Finally we checked by direct computation that the monomials in these lifts are linearly independent in weights up to $20$.

It is important to verify the correctness of this computation: incorrect computations at any step in the above argument are very likely to produce series which are algebraically independent, regardless of whether this is actually the case. To verify that the algorithm is correct, we computed not only the Eisenstein series above, but all Eisenstein series (for both cusps) of weights up to and including $20$. We expect, then, to find unique expressions for these additional Eisenstein series in terms of the generators above. Indeed, our method yields unique expressions for them; for example, we find \[\cE_{14,0} + \cE_{14,1} = 1330560 \cE_4^2 \cE_6 + 2640 \cE_4 (E_{10,0} + \cE_{10,1}) - 11088 \cE_6 (\cE_{8,0} + \cE_{8,1}). \qedhere \]
\end{proof}

From the above proposition, we derive a lower bound of $\dim M_k(\Gamma_{C_8})$ for $k\leq 18$. By direct calculation, we find that this lower bound coincides with the upper bound given in \S \ref{sec:upper bound}. As in Corollary \ref{cor:structure}, when $k\leq 18$, we have 
$$
\dim M_k(\Gamma_{C_8}) = \sum_{r=0}^{\infty} \dim J_{k-12r,D_8,r}^{\w ,\Orth(D_8)}, \quad k\leq 18.
$$
Moreover, for any $\phi\in J_{k,D_8,m}^{\Orth(D_8)}(q^m)$ with $k\leq 18$, there exists a modular form of weight $k$ on $\Gamma_{C_8}$ whose Fourier--Jacobi expansion begins $\phi \cdot \xi^m + O(\xi^{m+1})$. This confirms the existence of the modular forms required in Assumption \ref{assum} for $C_8$. We obtain the following theorem.

\begin{theorem}
The graded algebra of modular forms on $\Orth^+(\II_{2,2}\oplus D_8(-1))$ is freely generated by the $11$ orthogonal Eisenstein series $\cE_4,\cE_6,\cE_{8,0},\cE_{8,1},\cE_{10,0},\cE_{10,1},\cE_{12,0},\cE_{12,1},\cE_{14,0},\cE_{16,0},\cE_{18,0}$. 
\end{theorem}

\subsubsection{The other cases}
We have constructed the generators satisfying Assumption \ref{assum} for the root system $C_8$. 
We now construct the desired generators for other $C_n$, $D_n$ and $F_4$ as the pullbacks of the $C_8$ generators. We only do this for $D_8$ and $D_7$ because the other cases are similar. For the root systems $C_8$ and $D_8$, we can choose the generators of Weyl invariant Jacobi forms of the following form at the same time 
\begin{align*}
&D_8:& &\phi_{0,D_8,1}, \phi_{-2,D_8,1}, \phi_{-4,D_8,1}, \psi_{-8,D_8,1}, \phi_{-6,D_8,2}, \phi_{-8,D_8,2}, \phi_{-10,D_8,2}, \phi_{-12,D_8,2}, \phi_{-14,D_8,2}, \\
&C_8:& &\phi_{0,D_8,1}, \phi_{-2,D_8,1}, \phi_{-4,D_8,1}, \psi_{-8,D_8,1}^2, \phi_{-6,D_8,2}, \phi_{-8,D_8,2}, \phi_{-10,D_8,2}, \phi_{-12,D_8,2}, \phi_{-14,D_8,2}.
\end{align*}
The choice means that all generators of the shape $\phi$ are invariant but $\psi_{-8,D_8,1}$ is anti-invariant under the action of the odd sign change of coordinates of $D_8\otimes \CC$. From the relation of the two kinds of generators, we see at once that Assumption \ref{assum} holds for root system $D_8$. We then obtain the structure result for modular forms on $\Gamma_{D_8}=\widetilde{\Orth}^+(\II_{2,2}\oplus D_8(-1))$.

The root lattice $D_7$ is a sublattice of $D_8$. The generators of index 2 of $W(D_7)$-invariant Jacobi forms can be chosen as the pullbacks of the above $\phi_{-2k, D_8, 2}$. Note that the pullback of $\phi_{-14,D_8,2}$ to $D_7$ is equal to the square of the unique $W(D_7)$-invariant weak Jacobi form of weight $-7$ and index 1 up to a constant, and the pullback of $\psi_{-8,D_8,1}$ to $D_7$ is identically zero. The first Fourier--Jacobi coefficient of a pullback of an orthogonal modular form $F$ is equal to the pullback of the first Fourier--Jacobi coefficient of $F$.  We then construct the desired generators corresponding to basic Jacobi forms of index 2 as pullbacks.  Thus Assumption \ref{assum} is satisfied for $D_7$.  

We next consider the case of $F_4$.  For $W(F_4)$-invariant Jacobi forms,  the generators $\varphi_{-8,D_4,2}$ and $\varphi_{-12,D_4,3}$ can be constructed as polynomial combinations of the generators of $W(D_4)$-invariant Jacobi forms. Therefore the corresponding desired modular forms can be constructed as the same polynomial combinations of the basic orthogonal modular forms for $D_4$. Hence Assumption \ref{assum} holds for $F_4$.

\subsection{The case of \texorpdfstring{$E_6$}{E6}}
The root lattice $E_6$ has cyclic discriminant group of prime discriminant $3$. We let $\rho$ denote the attached (dual) Weil representation and write out modular forms taking values in $\rho$ in components: $F = (F_0,F_1,F_2)$. By \cite{BB03} there are isomorphisms in even weight: $$J_{2 \ast,E_6,1} = J_{2\ast,E_6,1}^{W(E_6)} \cong M_{2\ast-3}(\rho) \cong M_{2\ast-3}^+(\Gamma_0(3),\chi),$$ the latter space consisting of modular forms for the quadratic Nebentypus $\chi$ whose Fourier expansions are supported on quadratic residues. The latter isomorphism is simply $F \mapsto (F_0(3\tau)+F_1(3\tau)+F_2(3\tau))/3$; its inverse maps $f(\tau) = \sum_n c(n) q^n \in M_{2*-3}^+(\Gamma_0(3),\chi)$ to $$\Big(\sum_{n \equiv 0 \, (3)} c(n) q^{n/3}, \frac{1}{2} \sum_{n \equiv 1 \, (3)} c(n) q^{n/3}, \frac{1}{2} \sum_{n\equiv 1 \, (3)} c(n) q^{n/3} \Big).$$ Also, one obtains all even-weight modular forms for $\rho$ (and thus all odd-weight Jacobi forms) through the isomorphism $$M_*(\SL_2(\ZZ)) \stackrel{\sim}{\longrightarrow} M_{*+4}(\rho), \; \; f \mapsto (0,f \cdot \eta^8, -f \cdot \eta^8),$$ where $\eta(\tau) = q^{1/24} \prod_{n=1}^{\infty} (1 - q^n)$ as usual. These observations make it possible to compute quickly with Jacobi forms of index $E_6$.

For even $k \ge 4$ we denote by $\cE_k$ the orthogonal Eisenstein series, obtained by lifting the Jacobi Eisenstein series (which correspond as above to the usual Eisenstein series in $M_{k-3}^+(\Gamma_0(3),\chi)$). Also we let $\mathfrak{M}_7$ and $\mathfrak{M}_{15}$ denote the (unique up to scalar multiple) additive lifts in weights $7$ and $15$, which correspond to $1$ and $E_8=E_4^2$ in $M_*(\mathrm{SL}_2(\mathbb{Z}))$ under the isomorphism $$M_*(\mathrm{SL}_2(\mathbb{Z})) \stackrel{\sim}{\longrightarrow} M_{*+4}(\rho) \stackrel{\sim}{\longrightarrow} J_{*+7,E_6,1}.$$

\begin{proposition} 
The series $$\cE_4,\cE_6,\mathfrak{M}_7,\cE_{10},\cE_{12},\mathfrak{M}_{15},\cE_{16},\cE_{18},\cE_{24}$$ satisfy no algebraic relations in weights less than $40$.
\end{proposition}

\begin{proof} The argument is essentially the same as the $D_8$ lattice. We realized $E_6$ as $\mathbb{Z}^6$ with quadratic form given by the standard Cartan matrix $$\begin{psmallmatrix} 2 & 0 & -1 & 0 & 0 & 0 \\ 0 & 2 & 0 & -1 & 0 & 0 \\ -1 & 0 & 2 & -1 & 0 & 0 \\ 0 & -1 & -1 & 2 & -1 & 0 \\ 0 & 0 & 0 & -1 & 2 & -1 \\ 0 & 0 & 0 & 0 & -1 & 2 \end{psmallmatrix}$$ and evaluated the Jacobi forms above along the line through the norm $12$ vector $v = (3,2,0,1,1,1)$ (such that their Gritsenko lifts may naturally be interpreted as paramodular forms of level $12$). We found that the monomials in these Gritsenko lifts of weights less than $40$ are linearly independent, and verified the correctness of the algorithm by finding the Eisenstein series of weights $8,14,20$ as (unique) polynomials in the generators given above.
\end{proof}

Similarly to the case of $D_8$,  the above proposition implies that Assumption \ref{assum} holds for $E_6$ so we conclude the following result.

\begin{theorem}
The graded algebra of modular forms on $\widetilde{\Orth}^+(\II_{2,2}\oplus E_6(-1))$ is freely generated by the series $\cE_4,\cE_6,\mathfrak{M}_7,\cE_{10},\cE_{12},\mathfrak{M}_{15},\cE_{16},\cE_{18},\cE_{24}$.
\end{theorem}

Gritsenko and Nikulin constructed a cusp form of weight 120 on $\widetilde{\Orth}^+(\II_{2,2}\oplus E_6(-1))$ with the character $\det$ in \cite[Theorem 4.3]{GN18}. This form was constructed as the quasi-pullback of the Borcherds form on $\II_{2,26}$ (see \cite{Bor98}) and we denote it by $\Phi_{120,E_6}$.  It is a strongly 2-reflective Borcherds product; namely, its zero divisor is a sum of rational quadratic divisors associated to vectors of norm $-1$ with multiplicity one:
$$
\Div(\Phi_{120,E_6})=\sum_{\substack{l\in \II_{2,2}\oplus E_6(-1)\\ (l,l)=-2}} l^\perp.
$$

By \cite[Corollary 1.8]{GHS09}, $\widetilde{\Orth}^+(\II_{2,2}\oplus E_6(-1))$ has only one nontrivial character ($\det$) and $\widetilde{\SO}^+(\II_{2,2}\oplus E_6(-1))$ has no nontrivial characters. It is easy to check that every modular form on $\widetilde{\Orth}^+(\II_{2,2}\oplus E_6(-1))$ with character $\det$ vanishes on the 2-reflective divisor. Thus we have 
\begin{equation}
M_k(\widetilde{\Orth}^+(\II_{2,2}\oplus E_6(-1)),\det)=\Phi_{120,E_6}\cdot M_{k-120}(\widetilde{\Orth}^+(\II_{2,2}\oplus E_6(-1)).
\end{equation}

We can also construct  $\Phi_{120,E_6}$ using the Rankin-Cohen-Ibukiyama type differential operator (see \cite[Proposition 2.1]{AI05} for the original version for Siegel modular forms, and \cite[Proposition 2.14]{Klo05}) or \cite[Proposition 5.6]{DKW19} for the version for orthogonal modular forms). The Rankin-Cohen-Ibukiyama type differential operator of our nine generators gives a modular form of weight 120 with character $\det$, which must equal $\Phi_{120,E_6}$ up to a constant.

There is another interesting Borcherds product on $\widetilde{\Orth}^+(\II_{2,2}\oplus E_6(-1))$.  We use the invariant version of Borcherds product in the context of Jacobi forms (see \cite[Theorem 4.2]{Gri18}). The Borcherds product of the unique $W(E_6)$-invariant weak Jacobi form of weight 0 will give a modular form of weight 45 on $\widetilde{\Orth}^+(\II_{2,2}\oplus E_6(-1))$.

\subsection{The case of \texorpdfstring{$E_7$}{E7}}

The lattice index $E_7$ admits nonzero Jacobi forms only in even weight, and via the theta decomposition one obtains an isomorphism $$J_{2*,E_7,1} \stackrel{\sim}{\longrightarrow} M_{2*-7/2}(\rho) \stackrel{\sim}{\longrightarrow} M_{2* - 7/2}(\Gamma_0(4)),$$ first to modular forms which transform with respect to the Weil representation $\rho$ attached to the rank-one lattice generated by a vector of norm one, and then to the Kohnen plus spaces of level $\Gamma_0(4)$, i.e. modular forms $\sum_{n=0}^{\infty} c(n) q^n$ of half-integral weight for which $c(n) = 0$ unless $n \equiv 0,1$ mod $4$. The passage to Kohnen's plus space is through the correspondence $$\Big( \sum_{n \in \mathbb{N}} c(n) q^n, \sum_{n \in \mathbb{N} + 1/4} c(n) q^n \Big) \in M_*(\rho) \leftrightarrow \sum_{n \equiv 0, 1 \, (4)} c(n/4) q^n \in M_*^+(\Gamma_0(4)).$$

The Jacobi Eisenstein series of index $E_7$ and weight $k$ correspond through the previous paragraph to Cohen's \cite{Coh75} half-integral weight Eisenstein series. In particular their Fourier coefficients are generalized Hurwitz class numbers.  We denote by $\cE_k$ the orthogonal Eisenstein series obtained from these as lifts. (It is convenient to refer to $\theta(\tau) = 1 + 2q + 2q^4 + ...$ as the Cohen Eisenstein series of weight $1/2$, corresponding to the orthogonal form $\cE_4$.

\begin{proposition}
The series $$\cE_4,\cE_6,\cE_{10},\cE_{12},\cE_{14},\cE_{16},\cE_{18},\cE_{22},\cE_{24},\cE_{30}$$ satisfy no algebraic relations in weights less than $40$.
\end{proposition}

\begin{proof} 
This is the same argument that we used for the root lattices $D_8$ and $E_6$. We realize $E_7$ as $\mathbb{Z}^7$ with Gram matrix $$\begin{psmallmatrix} 2 & 0 & -1 & 0 & 0 & 0 & 0 \\ 0 & 2 & 0 & -1 & 0 & 0 & 0 \\ -1 & 0 & 2 & -1 & 0 & 0 & 0 \\ 0 & -1 & -1 & 2 & -1 & 0 & 0 \\ 0 & 0 & 0 & -1 & 2 & -1 & 0 \\ 0 & 0 & 0 & 0 & -1 & 2 & -1 \\ 0 & 0 & 0 & 0 & 0 & -1 & 2 \end{psmallmatrix}$$ and evaluated the Jacobi Eisenstein series in their abelian variable along the line through the norm $12$ vector $v = (3,2,0,1,1,1,1)$. We found that the monomials in the resulting Gritsenko lifts of weights less than $40$ are linearly independent, and we verified the correctness of the algorithm by actually computing the lifts of all Eisenstein series of weights up to $30$ and expressing them (uniquely) through the generators in the claim.
\end{proof}

We note that $W(E_7)=\Orth(E_7)$ and $\Orth^+(\II_{2,2}\oplus E_7(-1))=\widetilde{\Orth}^+(\II_{2,2}\oplus E_7(-1))$. 

Similarly, we have the following result.

\begin{theorem}
The graded algebra of modular forms on $\Orth^+(\II_{2,2}\oplus E_7(-1))$ is freely generated by the $10$ orthogonal Eisenstein series $\cE_4,\cE_6,\cE_{10},\cE_{12},\cE_{14},\cE_{16},\cE_{18},\cE_{22},\cE_{24},\cE_{30}$.
\end{theorem}

Gritsenko and Nikulin  also constructed a strongly 2-reflective cusp form of weight 165 on $\Orth^+(\II_{2,2}\oplus E_7(-1))$ with the character $\det$ in \cite[Theorem 4.3]{GN18}. We denote it by $\Phi_{165,E_7}$. This form is equal to the Rankin-Cohen-Ibukiyama type differential operator of our ten generators up to a constant.

By \cite[Corollary 1.8]{GHS09}, $\Orth^+(\II_{2,2}\oplus E_7(-1))$ has only one nontrivial character, namely $\det$, and $\SO^+(\II_{2,2}\oplus E_7(-1))$ has no nontrivial characters. Similarly,
\begin{equation}
M_k(\Orth^+(\II_{2,2}\oplus E_7(-1)),\det)=\Phi_{165,E_7}\cdot M_{k-165}(\Orth^+(\II_{2,2}\oplus E_7(-1))).
\end{equation}

There is also another interesting Borcherds product on $\Orth^+(\II_{2,2}\oplus E_7(-1))$.  The Borcherds product of the unique $W(E_7)$-invariant weak Jacobi form of weight 0 will give a modular form of weight 44 on $\Orth^+(\II_{2,2}\oplus E_7(-1))$.

We end this paper with two remarks.

\begin{remark}
The approach in the present paper does not apply to the other root systems in Wirthm\"{u}ller's theorem. The reason is that for other root systems there are some generators $\phi$ of index 1 of Weyl invariant weak Jacobi forms such that $\Delta\phi$ is not a holomorphic Jacobi form. In these cases the additive lift of $\Delta \phi$ is not holomorphic (see \cite[Theorem 14.3]{Bor98}).
\end{remark}

\begin{remark}
The bigraded ring of $W(E_8)$-invariant weak Jacobi forms is not a free algebra \cite{Wan18} so the methods of this paper also do not apply to the $E_8$-lattice. However, it was announced in \cite{HU14} that the graded ring of modular forms on $\Orth^+(\II_{2,2}\oplus E_8(-1))$ is freely generated by forms of weights 4, 10, 12, 16, 18, 22, 24, 28, 30, 36, 42. Assuming the structure result of \cite{HU14}, it was shown in \cite{DKW19} that the generators can be constructed as the additive lifts of Jacobi Eisenstein series.
\end{remark}

The proof of Theorem \ref{th:main} is finished. We formulate the groups $\Gamma_R$ and the weights of generators in the following theorem. The group $\Orth_1^+(\II_{2,2}\oplus D_4(-1))$ is defined in \S \ref{subsec:D}.

\begin{theorem}\label{th:generators}
For the arithmetic groups in the tables below, the graded rings of modular forms are free on generators of the given weights.

\begin{longtable}{|l|l|l|}
\caption{Some free algebras of orthogonal modular forms} \\
\hline
Root system &  Group &   Weights of generators                       \\ \hline
$A_1$ & $\Orth^+(\II_{2,2}\oplus A_1(-1))$& $4, 6, 10, 12$. \\ \hline
$A_2$ & $\widetilde{\Orth}^+(\II_{2,2}\oplus A_2(-1))$& $4, 6, 9, 10, 12$. \\ \hline
$B_2$ & $\Orth^+(\II_{2,2}\oplus 2A_1(-1))$& $4, 6, 8, 10, 12$. \\ \hline
$G_2$ & $\Orth^+(\II_{2,2}\oplus A_2(-1))$& $4, 6, 10, 12, 18$. \\ \hline
$A_3$ & $\widetilde{\Orth}^+(\II_{2,2}\oplus A_3(-1))$& $4, 6, 8, 9, 10, 12$. \\ \hline
$B_3$ & $\Orth^+(\II_{2,2}\oplus 3A_1(-1))$& $4, 6, 6, 8, 10, 12$. \\ \hline
$C_3$ & $\Orth^+(\II_{2,2}\oplus A_3(-1))$& $4, 6, 8, 10, 12, 18$. \\ \hline
$A_4$ & $\widetilde{\Orth}^+(\II_{2,2}\oplus A_4(-1))$& $4, 6, 7, 8, 9, 10, 12$. \\ \hline
$B_4$ & $\Orth^+(\II_{2,2}\oplus 4A_1(-1))$& $4, 4, 6, 6, 8, 10, 12$. \\ \hline
$C_4$ & $\Orth_1^+(\II_{2,2}\oplus D_4(-1))$& $4, 6, 8, 10, 12, 16, 18$. \\ \hline
$D_4$ & $\widetilde{\Orth}^+(\II_{2,2}\oplus D_4(-1))$& $4, 6, 8, 8, 10, 12, 18$. \\ \hline
$F_4$ & $\Orth^+(\II_{2,2}\oplus D_4(-1))$& $4, 6, 10, 12, 16, 18, 24$. \\ \hline
\end{longtable}
\newpage
\begin{longtable}{|l|l|l|}
\caption{Some free algebras of orthogonal modular forms, continued} \\
\hline
 Root system &  Group &   Weights of generators                       \\ \hline
$A_5$ & $\widetilde{\Orth}^+(\II_{2,2}\oplus A_5(-1))$& $4, 6, 6, 7, 8, 9, 10, 12$. \\ \hline
$C_5$ & $\Orth^+(\II_{2,2}\oplus D_5(-1))$& $4, 6, 8, 10, 12, 14, 16, 18$. \\ \hline
$D_5$ & $\widetilde{\Orth}^+(\II_{2,2}\oplus D_5(-1))$& $4, 6, 7, 8, 10, 12, 16, 18$. \\ \hline
$A_6$ & $\widetilde{\Orth}^+(\II_{2,2}\oplus A_6(-1))$& $4, 5, 6, 6, 7, 8, 9, 10, 12$. \\ \hline
$C_6$ & $\Orth^+(\II_{2,2}\oplus D_6(-1))$& $4, 6, 8, 10, 12, 12, 14, 16, 18$. \\ \hline
$D_6$ & $\widetilde{\Orth}^+(\II_{2,2}\oplus D_6(-1))$& $4, 6, 6, 8, 10, 12, 14, 16, 18$. \\ \hline
$E_6$ & $\widetilde{\Orth}^+(\II_{2,2}\oplus E_6(-1))$& $4,6,7,10,12,15,16,18,24$. \\ \hline
$A_7$ & $\widetilde{\Orth}^+(\II_{2,2}\oplus A_7(-1))$& $4, 4, 5, 6, 6, 7, 8, 9, 10, 12$. \\ \hline
$C_7$ & $\Orth^+(\II_{2,2}\oplus D_7(-1))$& $4, 6, 8, 10, 10, 12, 12, 14, 16, 18$. \\ \hline
$D_7$ & $\widetilde{\Orth}^+(\II_{2,2}\oplus D_7(-1))$& $4, 5, 6, 8, 10, 12, 12, 14, 16, 18$. \\ \hline
$E_7$ & $\Orth^+(\II_{2,2}\oplus E_7(-1))$ & $4,6,10,12,14,16,18,22,24,30$. \\ \hline
$C_8$ & $\Orth^+(\II_{2,2}\oplus D_8(-1))$& $4, 6, 8, 8, 10, 10, 12, 12, 14, 16, 18$. \\ \hline
$D_8$ & $\widetilde{\Orth}^+(\II_{2,2}\oplus D_8(-1))$& $4, 4, 6, 8, 10, 10, 12, 12, 14, 16, 18$. \\ \hline

\end{longtable}

\end{theorem}

\bigskip

\noindent
\textbf{Acknowledgements} 
H. Wang would like to thank Zhiwei Zheng for numerous stimulating discussions on related topics, and he is grateful to Max Planck Institute for Mathematics in Bonn for its hospitality and financial support. B. Williams is supported by a fellowship of the LOEWE research center ``Uniformized Structures in Arithmetic and Geometry". The authors thank Eberhard Freitag, Valery Gritsenko and Aloys Krieg for their comments.

\bibliographystyle{amsalpha}

\end{document}